\documentclass[12pt]{article}

 \usepackage{hyperref}

\usepackage{amsfonts, amsmath, amssymb, amsthm}
\usepackage{a4}

\usepackage{upgreek}

\DeclareMathOperator{\Ker}{Ker}
\DeclareMathOperator{\LinearImage}{Im}
\DeclareMathOperator{\rank}{rank}

\newtheorem{theorem}{Theorem}

\theoremstyle{remark}
\newtheorem*{remark}{Remark}
\newtheorem*{ir}{Important remark}
\newtheorem*{xmp}{Example}

\theoremstyle{definition}
\newtheorem{cnd}{Condition}

\author{I.G. Korepanov}
\title{Nonconstant hexagon relations and their cohomology}
\date{March 2018--January 2019}

\begin{document}

\sloppy

\maketitle

\begin{abstract}
A construction of hexagon relations---algebraic realizations of four-dimen\-sional Pachner moves---is proposed. It goes in terms of ``permitted colorings'' of 3-faces of pentachora (4-simplices), and its main feature is that the set of permitted colorings is \emph{nonconstant}---varies from pentachoron to pentachoron. Further, a cohomology theory is formulated for these hexagon relations, and its nontriviality is demonstrated on explicit examples.
\end{abstract}

\section{Introduction}\label{s:i}

There are some algebraic and combinatorial structures in topology and mathematical physics that, being already interesting in themselves, admit also very nontrivial cohomology theories, built over them in a natural way. The most studied example seems to be \emph{quandles}: they give, by themselves, invariants of \emph{knots} as well as their higher-dimen\-sional counterparts, and \emph{quandle cohomology}~\cite{CKS} gives a refined, more sensitive versions of these invariants.

Quandles are a particular case of \emph{Yang--Baxter maps}~\cite{dri,wei}, and these are further generalized to \emph{$n$-simplex maps}, or $n$-simplex relations, that also have a nontrivial cohomology~\cite{KST}. This cohomology is a very interesting direction of research that promises major progress in building and understanding `exactly solvable models' in statistical physics and quantum field theory.

These days, one more theory is appearing, if not completely parallel to, then certainly inspired by $n$-simplex cohomology: cohomology of algebraic realizations of Pachner moves~\cite{Pachner,Lickorish}---elementary rebuildings of a triangulation of a given piecewise linear (PL) manifold. The present paper is about a particular case of this cohomology, hopefully related to interesting invariants of four-dimen\-sional PL manifolds, or rather pairs ``such manifold, its middle cohomology class''. The main feature of our four-dimen\-sional Pachner move realizations, also called `hexagon relations', is that they are `nonconstant'---depending on parameters, in the same sense as there are, for instance, constant Yang--Baxter (YB) relations and YB relations depending on parameters.

We recall very briefly the definitions of the mentioned two sorts of cohomology here in Subsections \ref{ss:rs} and~\ref{ss:rg} (more on Pachner moves will be said in Section~\ref{s:g}). Then, in Subsection~\ref{ss:in} we explain the idea of `nonconstant' cohomology, and in Subsection~\ref{ss:ia} we outline the contents of the rest of the paper.

\subsection[Reminder of $n$-simplex cohomology]{Reminder of $\boldsymbol n$-simplex cohomology}\label{ss:rs}

The objects involved in $n$-simplex relations and their cohomology are as follows. First, an $n$-cube~$I^n$ is considered, whose $(n-1)$-faces can be `colored': each face is supplied with an element of a given `set of colors'~$X$. A coloring of the whole $n$-cube is thus an element of the direct product $X^{\times 2n}$ of $2n$ copies of~$X$. Then, a subset $R\subset X^{\times 2n}$ of `permitted colorings' is defined. The notion of permitted coloring can be extended also to an $N$-cube~$I^N$, with $N\ge n$: this is such a coloring of its $(n-1)$-faces that its restriction on any $n$-face is permitted. $N$-chains are formal linear combinations of permitted colorings of~$I^N$ with coefficients in~$\mathbb Z$ or other abelian group. The differential~$d$ acts on an $N$-chain~$c$, for $N\ge n$, as follows:
\begin{equation}\label{dc}
d(c)=\sum_{k=1}^N \bigl( c|_k^f-c|_k^r \bigr),
\end{equation}
It is assumed in~\eqref{dc} that all $(N-1)$-faces have been divided into pairs of \emph{parallel} faces. One face in each pair is called \emph{front} while the other---\emph{rear}, and $c|_k^f$ and $c|_k^r$ are the restrictions of~$c$ onto the front and rear faces in the $k$-th pair.

\begin{ir}
This construction implies, of course, that all such details as the rule describing the identification of any $n$-face with the `standard' $n$-cube, as well as the rule for choosing `front' and `rear' faces, must be provided. All these details can be found in~\cite{KST}.
\end{ir}

Then, $n$-simplex \emph{cohomology} is built from the homology with differential~\eqref{dc} in a standard way. Note that it deals with the infinite sequence of cubes:
\begin{equation}\label{cubes}
I^{n-1},\quad I^{n},\quad I^{n+1},\; \ldots,\;\; I^N,\; \ldots
\end{equation}

\subsection[Reminder of $(n+2)$-gon cohomology]{Reminder of $\boldsymbol{(n+2)}$-gon cohomology}\label{ss:rg}

On the other hand, there is a cohomology theory, called somewhat loosely \emph{$(n+2)$-gon cohomology}, that deals likewise with an infinite sequence of \emph{simplices}:
\begin{equation}\label{simplices}
\Delta^{n-1},\quad \Delta^{n},\quad \Delta^{n+1},\; \ldots,\;\; \Delta^N,\; \ldots
\end{equation}
Here, again, $(n-1)$-faces can be colored using some color set~$X$, and their colors around any $n$-simplex ($n$-face of an $N$-simplex, $N\ge n$) must be consistent, that is, belong to a given `permitted' subset~$R\subset X^{\times(n+1)}$. Then, $N$-chains are formal linear combinations of permitted colorings of~$\Delta^N$, and the differential~$d$ acts on an $N$-chain~$c$ in the following non-surprising way:
\begin{equation}\label{dg}
d(c)=\sum_{k=0}^N (-1)^k c|_k \,. 
\end{equation}
In~\eqref{dg}, it is assumed that the vertices of~$\Delta^N$ are numbered from $0$ to~$N$, and $c|_k$ is the restriction of~$c$ onto the $(N-1)$-face lying \emph{opposite} vertex~$k$; all these $(N-1)$-faces are then identified with the `standard'~$\Delta^{N-1}$ with vertices $0,\ldots,N-1$ in such way that the \emph{order} of vertices is conserved.

Cohomology is then constructed in a standard way, see~\cite{cubic} for details and applications.

\subsection{Nonconstant cohomology}\label{ss:in}

In both above constructions of $n$-simplex and $(n+2)$-gon cohomology, there is the following common feature: there is one and the same---\emph{constant}---subset~$R$ of permitted colorings for any $n$-cube or, respectively, $n$-simplex (which implies of course that there is a procedure for identifying this cube or simplex with one `standard' cube or simplex).

\begin{remark}
It may make sense to draw the reader's attention to the terminology: $n$-simplex relations and cohomology deal with \emph{cubes}, while $(n+2)$-gon relations and cohomology deal with \emph{simplices}!
\end{remark}

On the other hand, there has been shown~\cite{full-nonlinear} that there exist interesting \emph{hexagon} (\,=\,6-gon) relations, at least in their `fermionic' version, naturally parameterized by a simplicial \emph{2-cocycle}~$\omega$ given on the 5-simplex~$\Delta^5$. This means that each of the six 4-faces of~$\Delta^5$ has its own `fermionic analogue' of the set of permitted colorings, depending on the restriction of~$\omega$ on this 4-face. It was further shown in~\cite{gcoi,gcoii} that invariants of a \emph{pair} ``PL 4-manifold, its middle cohomology class'' can be built based on these hexagon relations. It must be noted that these `fermionic' invariants appear to be, at this stage, rather difficult computationally; also, no fermionic analogue of hexagon \emph{cohomology} has been formulated as yet.

In any case, with the growing computational powers and, on the other hand, further development of the theory, these invariants will, hopefully, be investigated. Also, it looks reasonable to study corresponding `bosonic' structures, that is, dealing with usual commuting variables rather than anticommuting Grassmann variables appearing in a fermionic theory.

We will see in this paper that there is, indeed, an interesting way of assigning \emph{different} subsets~$R_u$ of permitted colorings to different 4-simplices (pentachora)~$u$, parameterized again by the values of a simplicial 2-cocycle~$\omega$, and, moreover, nontrivial hexagon cohomology does appear in our `bosonic' theory! 

\subsection{What this paper is about}\label{ss:ia}

This paper mainly deals with algebraic structures occurring within (the boundary~$\partial \Delta^5$ of) just one 5-simplex~$\Delta^5$, although we always keep in mind potential applications to four-dimen\-sional piecewise linear topology. Our algebraic structures include a `nonconstant hexagon relation' and some bilinear forms that depend on permitted colorings and can be interpreted as `nonconstant cocycles'.

Below,
\begin{itemize}\itemsep 0pt
 \item in Section~\ref{s:g}, we explain, in a general setting, the relations between the ``nonconstant hexagon'' and four-dimen\-sional Pachner moves,
 \item in Section~\ref{s:h}, we explain, again in a general setting, the idea of ``nonconstant hexagon cohomology'',
 \item in Section~\ref{s:f}, we present a quite general construction of ``linear nonconstant hexagon'', starting from ``edge functionals'' satisfying given linear relations in each vertex,
 \item in Section~\ref{s:exotic}, we introduce ``edge vectors'' --- objects, in a sense, dual to edge functionals. On this base, we build an ``exotic'' chain complex, with its own homology (not to be confused with the hexagon cohomology), first for the boundary of a single 5-simplex, and then for a general simplicial complex,
 \item in Section~\ref{s:fi}, we specialize our constructions of linear nonconstant hexagon and its cohomology for a very interesting ``infinitesimal'' case---and this is exactly the case where the 2-cocycle~$\omega$ mentioned in Subsection~\ref{ss:in} appears. Among other things, we write out symmetric bilinear forms that represent, in this ``infinitesimal'' case, nontrivial hexagon 3- and 4-cocycles,
 \item in Section~\ref{s:fermionic}, we find somewhat unexpected connections of the present work with an earlier work on ``free fermions on four-dimen\-sional PL manifolds'',
 \item and finally, in Section~\ref{s:d}, we discuss our results and prospects for further research.
\end{itemize}

\section{Nonconstant set theoretic hexagon relations: generalities}\label{s:g}

Piecewise linear manifold invariants can be built if we find \emph{algebraic realizations of Pachner moves}---algebraic formulas whose structure corresponds naturally to these moves. This has been done very successfully in the three-dimen\-sional case~\cite{3d-tqft}, and there also some known four-dimen\-sional invariants~\cite{4d-tqft}. In this paper, we consider algebraic realizations of four-dimen\-sional Pachner moves, or hexagon relations. The closest already known analogues of our relations are those in~\cite{cubic}, except that there they are \emph{constant}.

This section mainly gives some general combinatorial, or `set theoretic', definitions.

\subsection{Pachner moves}\label{ss:P}

Consider a 5-simplex $\Delta^5=123456$ (i.e., whose vertices are numbered from 1 to~6). Its boundary~$\partial \Delta^5$ consists of six pentachora (\,=\,4-simplices). Imagine that $k$ of these pentachora, $1\le k\le 5$, enter in a triangulation of a four-dimen\-sional PL manifold~$M$. Then we can replace them with the remaining $6-k$ pentachora, without changing~$M$. This is called four-dimen\-sional Pachner move, and there are five kinds of them: 1--5, 2--4, 3--3, 4--2, and 5--1; here the number before the dash is~$k$, while the number after the dash is, of course, $6-k$.

We sometimes call the initial configuration---cluster of~$k$ pentachora---the \emph{left-hand side} (l.h.s.)\ of the Pachner move, while its final configuration---cluster of~$6-k$ pentachora---its \emph{right-hand side} (r.h.s.).
 
\subsection{Ordering of vertices}\label{ss:o}

We assume that all vertices in every triangulation we are considering are \emph{ordered}. Typically, our triangulated 4-manifold will be just~$\partial\Delta^5$, whose vertices are ordered due to the fact that they are \emph{numbered} from~1 to their total number~6.

\begin{remark}
Of course our future manifold invariants must be independent of any such ordering. This is why `full hexagon' is needed, see Subsection~\ref{ss:full}.
\end{remark}

When we speak about an individual pentachoron, like in formula~\eqref{R12345} below, we may denote it as~$12345$---but we keep in mind that our constructions or/and statements are also valid for any pentachoron~$ijklm$, \ $i<j<k<l<m$, if we do the obvious replacements $1\mapsto 1$, \ldots, $5\mapsto m$.

Similarly, the description of Pachner moves in the above Subsection~\ref{ss:P} stays valid, of course, for any vertices $i,\ldots,n$ instead of~$1,\ldots,6$.

\subsection{Permitted colorings}\label{ss:perm}

Let a set~$X$ be given, called \emph{the set of colors}. We will assign a color $\mathsf x\in X$ to each \emph{3-face} of a simplicial complex such as~$\Delta^5$ or triangulated manifold~$M$. Not all colorings, however, are \emph{permitted}.

For one pentachoron~$u$, permitted colorings are determined by definition by a given subset~$R_u$ in the Cartesian product of five copies~$X_t$ of~$X$ corresponding to the 3-faces $t\subset u$. For instance, if $u=12345$, then there must be given a subset
\begin{equation}\label{R12345}
R_{12345}\subset X_{2345}\times X_{1345}\times X_{1245}\times X_{1235}\times X_{1234}.
\end{equation}
The main difference between the present paper and the previous papers~\cite{cubic,constant} is that now~$R_u$, for different pentachora~$u$, are \emph{not} supposed to be copies of one another.

\begin{remark}
Also, we don't assume $X$ to be necessarily a \emph{finite} set at this moment. Other important cases are listed in the next Subsection~\ref{ss:full}.
\end{remark}

For a cluster~$C$ of pentachora obtained by gluing them along their 3-faces---such as the l.h.s.\ or r.h.s.\ of a Pachner move, or even a `big' triangulated 4-manifold---permitted coloring is by definition such a coloring of all 3-faces (including inner faces where gluing has been done) whose restriction onto each pentachoron is permitted. We denote the set of permitted colorings of~$C$ as $R_C\subset \prod_{t\subset C} X_t$.

\subsection{Full set theoretic hexagon}\label{ss:full}

Take any subcomplex~$C\subset\partial \Delta^5$ containing $k$ pentachora, \ $1\le k\le 5$, and its complementary subcomplex~$\bar C$ containing the remaining $6-k$ pentachora. First, we introduce the following condition on the sets~$R_u$.

\begin{cnd}\label{cnd:CbarC}
The restrictions of permitted colorings of any such~$C$ onto the common boundary $\partial C=\partial \bar C$ make up the same set of colorings of this common boundary as the restrictions of permitted colorings of~$\bar C$.
\end{cnd}

Note that a coloring of $\partial C=\partial \bar C$ may \emph{not} determine uniquely the colorings of (the inner parts of) $C$ or/and~$\bar C$. To take this fact into account, we would like to introduce \emph{multiplicities}, in order to measure the (sometimes infinite) number of permitted colorings of $C$ or~$\bar C$ corresponding to a fixed coloring of their common boundary. We will actually deal with situations which fall within the following possibilities:
\begin{enumerate}\itemsep 0pt
 \item\label{i:xf} the color set~$X$ is \emph{finite}. Then we define multiplicities simply as the mentioned (natural) numbers of permitted colorings, given a fixed coloring of the boundary,
 \item\label{i:md} $X$ is a \emph{finite-dimensional vector space} over some field~$F$, and the permitted colorings of $C$ and~$\bar C$ form, for a fixed boundary coloring, \emph{affine subspaces} in a suitable direct sum $X\oplus \dots \oplus X$ (each direct summand corresponds to a relevant tetrahedron). Then we define multiplicities as dimensions of the mentioned affine subspaces,
 \item\label{i:xg} $X$ is a \emph{finitely generated abelian group}, and the permitted colorings of $C$ and~$\bar C$ form, for a fixed boundary coloring, \emph{cosets} of some abelian groups $H$ and~$\bar H$ (everything happens within a suitable direct sum $X\oplus \dots \oplus X$). Then we define multiplicities as isomorphism classes of $H$ and~$\bar H$.
\end{enumerate}
In all these cases, the multiplicity will not depend on a specific boundary coloring.

\begin{cnd}\label{cnd:ak}
There are fixed multiplicities~$a_k$, \ $1\le k\le 5$, such that for any chosen coloring of $\partial C=\partial \bar C$, the multiplicity of colorings of~$C$ is exactly~$a_k$, and the multiplicity of colorings of~$\bar C$ is exactly~$a_{6-k}$.
\end{cnd}

If both Conditions \ref{cnd:CbarC} and~\ref{cnd:ak} hold, we say that the \emph{full set theoretic hexagon}, or simply \emph{full hexagon}, is satisfied.

\begin{ir}
It may happen of course that $X$ matches more than one of the above cases \ref{i:xf}--\ref{i:xg}, so we will have more than one different definitions of~$a_k$ for the same~$X$. This will, however, bring us no problem because whenever we meet such situation in this paper, Condition~\ref{cnd:ak} will hold for \emph{any} definition.
\end{ir}

\section{Nonconstant hexagon cohomology: generalities}\label{s:h}

Our definition of nonconstant hexagon cohomology will depend on a chosen simplicial complex~$K$. In principle, $K$ can be of any dimension, although the main work in this paper will take place in the standard 5-simplex~$K=\Delta^5$.

Suppose that every 3-simplex~$t\subset K$ is colored by some element $\mathsf x_t\in X$ of the set of colors, and that, and a subset~$R_u$ of permitted colorings is defined in the set of all colorings of every 4-simplex~$u$, as in Subsection~\ref{ss:perm}.

We also define the set of permitted colorings for any simplex of~$K$ of dimension~$>4$: the coloring is permitted provided its restrictions on all 4-faces of that simplex are permitted. As for an individual 3-simplex, \emph{all} its colorings $\mathsf x\in X$ are permitted.

The set of all permitted colorings of an $n$-simplex $i_0\dots i_n$ will be denoted~$\mathfrak C_{i_0\dots i_n}$. Here and throughout the paper, we assume by default that the vertices of any simplex are ordered: $i_0<\ldots <i_n$.

Let an abelian group~$G$ be given. By definition, an \emph{$n$-cochain}~$\mathfrak c$ taking values in~$G$, for $n\ge 3$, consists of arbitrary mappings
\begin{equation}\label{nc}
\mathfrak c_{i_0\dots i_n}\colon\;\,\mathfrak C_{i_0\dots i_n} \to G
\end{equation}
for \emph{all} $n$-simplices $\Delta^n=i_0\dots i_n\subset K$.

The \emph{coboundary}~$\delta \mathfrak c$ of~$\mathfrak c$ consists then of mappings
\begin{equation}\label{cb}
(\delta \mathfrak c)_{i_0\dots i_{n+1}} = \sum_{k=0}^{n+1} (-1)^k\, \mathfrak c_{i_0\dots \widehat{i_k} \dots i_{n+1}}.
\end{equation}

Some variations of the cochain definition~\eqref{nc} are also of interest for us. For instance, the paper~\cite{cubic} (although devoted to \emph{constant} hexagon cohomology) suggests that \emph{homogeneous polynomials} of a given degree may be used instead of functions~\eqref{nc}---of course, in a situation where the notion of polynomial in the variables determining a permitted coloring makes sense.

Of special interest for us in this paper will be \emph{symmetric bilinear cochains}. These occur in the case where permitted colorings form a \emph{module over a commutative ring\/~$R$}---typically either a vector space or an abelian group ($\mathbb Z$-module). A symmetric bilinear cochain depends on \emph{two} (independent from each other) permitted colorings; the definition~\eqref{nc} changes accordingly into
\begin{equation}\label{bc}
\mathfrak c_{i_0\dots i_n}\colon\quad \mathfrak C_{i_0\dots i_n} \times \mathfrak C_{i_0\dots i_n} \to R,
\end{equation}
where $\mathfrak c_{i_0\dots i_n}$ must be, first, bilinear and second, symmetric with respect to interchanging its two arguments.

Our `nonconstant hexagon cohomology' is, in any of the above cases, the cohomology of the following \emph{hexagon cochain complex}:
\begin{equation}\label{hcc}
0 \to C^3 \stackrel{\delta}{\to} C^4 \stackrel{\delta}{\to} C^5 \stackrel{\delta}{\to} \dots\, ,
\end{equation}
where $C^n$ means the group of all $n$-cochains.

\section{Edge functionals and linear nonconstant hexagon: general case}\label{s:f}

\subsection{Edge functionals in a single pentachoron}\label{ss:p}

We are going now to consider a very interesting case where the color set~$X$ is a \emph{two-dimen\-sional vector space} over a field~$F$---so we write elements of~$X$ as two-columns:
\begin{equation}\label{tc}
X\ni \mathsf x = \begin{pmatrix}x\\ y\end{pmatrix},\quad x,y\in F,
\end{equation}
and permitted colorings of a pentachoron are singled out by \emph{linear relations}. Namely, there is one linear relation associated with each pentachoron \emph{edge}~$ij$, stated as the vanishing of a linear \emph{edge functional}~$\upphi_{ij}$, and this~$\upphi_{ij}$ can depend only on the colors of the tetrahedra containing the edge: $t\supset ij$. Edge functionals are defined for \emph{unoriented} edges: $\upphi_{ij}=\upphi_{ji}$.

The set of permitted colorings for a pentachoron~$u$ is, by definition, the intersection of kernels of all edge functionals. We will prefer to denote this set, in this `linear' case, as $V_u$ rather than~$R_u$:
\begin{equation}\label{pk}
V_u = \bigcap_{ij\subset u} \Ker \upphi_{ij}.
\end{equation}

\begin{ir}
The field~$F$ can be of any characteristic, but it must be big enough to contain elements that do not satisfy any `casual' algebraic relation. Below, we call such elements \emph{generic}, or lying in the \emph{general position}. What `general position' means, will always be clear from the context.
\end{ir}

The colorings of a tetrahedron~$t$ being written as two-columns~\eqref{tc}, we can write the restriction of~$\upphi_{ij}$ onto~$t$ as a two-\emph{row}:
\begin{equation}\label{phi_ij-t}
\upphi_{ij}|_t = \begin{pmatrix} \phi_{t,ij}^{(1)} & \phi_{t,ij}^{(2)} \end{pmatrix}.
\end{equation}

\begin{xmp}
In these notations, the relation $\upphi_{12}=0$ in pentachoron~$12345$ looks as follows:
\begin{multline}\label{12}
\phi_{1234,12}^{(1)}x_{1234}+\phi_{1234,12}^{(2)}y_{1234} + \phi_{1235,12}^{(1)}x_{1235}
+\phi_{1235,12}^{(2)}y_{1235} + \phi_{1245,12}^{(1)}x_{1245}\\
+\phi_{1245,12}^{(2)}y_{1245} = 0.
\end{multline}
\end{xmp}

Further, we postulate also linear relations \emph{between edge functionals}, each such relation associated with a vertex~$i$:
\begin{equation}\label{g-phi}
\sum_{\substack{\text{all four edges }ij\\ \text{for each fixed }i}} \gamma_{ij} \upphi_{ij} = 0.
\end{equation}
As we are going to see (in explicit expressions~\eqref{phi_1234-gen}), quantities~$\gamma_{ij}$ will be the parameters that determine our edge functionals up to `gauge transformations'---linear changes of bases of two-dimen\-sional coloring spaces in each separate tetrahedron.

\begin{ir}
In contrast with what we said about $\upphi_{ij}=\upphi_{ji}$, values $\gamma_{ij}$ and~$\gamma_{ji}$ in~\eqref{g-phi} are \emph{not} related to each other.
\end{ir}

\begin{xmp}
The restriction of~\eqref{g-phi} onto tetrahedron~$1234$ looks as follows:
\begin{equation}\label{gamma-phi}
\begin{pmatrix} \gamma_{12} & \gamma_{13} & \gamma_{14} & 0 & 0 & 0 \\
                \gamma_{21} & 0 & 0 & \gamma_{23} & \gamma_{24} & 0 \\
                0 & \gamma_{31} & 0 & \gamma_{32} & 0 & \gamma_{34} \\
                0 & 0 & \gamma_{41} & 0 & \gamma_{42} & \gamma_{43}
\end{pmatrix}
\begin{pmatrix} \upphi_{12} \\ \upphi_{13} \\ \upphi_{14} \\
 \upphi_{23} \\ \upphi_{24} \\ \upphi_{34} \end{pmatrix}_{1234} = 0.
\end{equation}
Here the subscript~$1234$ at the column of~$\upphi$'s means of course their restrictions onto tetrahedron~$1234$.
\end{xmp}

It follows from~\eqref{gamma-phi} and similar relations for any tetrahedron~$t$ that, for \emph{generic} coefficients~$\gamma_{ij}$, the six-column of the restrictions~$\upphi_{ij}$---which is actually a $(6\times 2)$-matrix according to~\eqref{phi_ij-t}---is determined by values~$\gamma_{ij}$ uniquely up to a change of basis in the two-dimen\-sional space of tetrahedron~$t$ colors.

\begin{xmp}
A simple calculation shows that, changing (if needed) the basis in the space of tetrahedron~$t=1234$ colorings, we can bring the $\upphi$-column in~\eqref{gamma-phi} to the following form:
\begin{equation}\label{phi_1234-gen}
\begin{pmatrix} \upphi_{12} \\ \upphi_{13} \\ \upphi_{14} \\
 \upphi_{23} \\ \upphi_{24} \\ \upphi_{34} \end{pmatrix}_{1234} =
\setlength\arraycolsep{1em}
\begin{pmatrix}
0 & -\gamma_{13} ( \gamma_{24} \gamma_{32} \gamma_{43}+\gamma_{23} \gamma_{34} \gamma_{42}) \\
\gamma_{14} ( \gamma_{24} \gamma_{32} \gamma_{43}+\gamma_{23} \gamma_{34} \gamma_{42}) & \gamma_{12} ( \gamma_{24} \gamma_{32} \gamma_{43}+\gamma_{23} \gamma_{34} \gamma_{42}) \\
-\gamma_{13} ( \gamma_{24} \gamma_{32} \gamma_{43}+\gamma_{23} \gamma_{34} \gamma_{42}) & 0\\
-\gamma_{24} ( \gamma_{14} \gamma_{31} \gamma_{43}+\gamma_{13} \gamma_{34} \gamma_{41}) & \gamma_{13} \gamma_{21} \gamma_{34} \gamma_{42}-\gamma_{12} \gamma_{24} \gamma_{31} \gamma_{43}\\
\gamma_{23} ( \gamma_{14} \gamma_{31} \gamma_{43}+\gamma_{13} \gamma_{34} \gamma_{41}) & ( \gamma_{13} \gamma_{21} \gamma_{32}+\gamma_{12} \gamma_{23} \gamma_{31})  \gamma_{43}\\
\gamma_{13} \gamma_{24} \gamma_{32} \gamma_{41}-\gamma_{14} \gamma_{23} \gamma_{31} \gamma_{42} & -( \gamma_{13} \gamma_{21} \gamma_{32}+\gamma_{12} \gamma_{23} \gamma_{31})  \gamma_{42}
\end{pmatrix}.
\end{equation}
\end{xmp}

Let~$u$ be an \emph{oriented} pentachoron, and $t=ijkl\subset u$ be its face. We \emph{define} $\upphi_{ij}|_t$~\eqref{phi_ij-t} as the corresponding row in~\eqref{phi_1234-gen}, with the obvious change $1\mapsto i$, \ldots, $4\mapsto l$, and multiplied by
\begin{equation}\label{eps}
\epsilon_t^{(u)} = \pm 1.
\end{equation}
Here plus is taken if the orientation of~$t$ given by the order $ijkl$ of vertices coincides with its orientation induced from~$u$, and minus is taken otherwise.

The space of all colorings of 5 tetrahedra---faces of our pentachoron~$u$---is $2\times 5=10$-dimen\-sional. There are then 10 relations of type~\eqref{12} (because there are 10 edges), and 5 relations~\eqref{g-phi} between these relations (because there are 5 vertices)---so, there remain $10-5=5$ independent relations. Hence, the space of \emph{permitted} colorings of a pentachoron is $10-5=5$-dimen\-sional.

\begin{ir}
We remind the reader that coefficients~$\gamma_{ij}$ are here \emph{generic}. Interestingly, there is also a remarkable \emph{limiting case} (formal limit in the case of a finite characteristic) which will be considered in Section~\ref{s:fi}.
\end{ir}

\subsection{Permitted boundary colorings for Pachner moves}\label{ss:bP}

The fundamental property of edge functionals~$\upphi_{ij}$ defined in Subsection~\ref{ss:p} is that they behave nicely when we glue pentachora together along their 3-faces.

\begin{theorem}\label{th:bc}
Let there be a cluster of 2, 3, 4 or~5 pentachora, such as take part in Pachner moves. Let generic coefficients~$\gamma_{ij}$ be given for each ordered pair $(i,j)$, \ $i\ne j$, of vertices in this cluster. Then, the colors of \emph{inner} tetrahedra can be excluded from the linear dependencies, in the sense that there are the following dependencies involving only \emph{boundary} (\,=\,outer) edges and tetrahedra:
\begin{equation}\label{phi-b}
\upphi_{ij}^{\mathrm{boundary}}=0 \quad \text{for all boundary edges }ij,
\end{equation}
with all components~$\upphi_{ij}^{\mathrm{boundary}}|_t$ given, again for \emph{boundary} tetrahedra~$t$, by the same definition as in the paragraph following~\eqref{phi_1234-gen}, with the signs~\eqref{eps} replaced with the signs~$\epsilon_t^{\mathrm{boundary}}$ reflecting the orientation of\/~$t$ compared to the orientation of the boundary. Moreover, boundary edge functionals satisfy the following version of linear relations~\eqref{g-phi}:
\begin{equation}\label{gp}
\sum_{\substack{\mathrm{all\; boundary\; edges\; }ij\\ \mathrm{for\; each\; boundary\; }i}} \gamma_{ij} \upphi_{ij}^{\mathrm{boundary}} = 0.
\end{equation}
\end{theorem}

\begin{proof}
Let $t$ be an \emph{inner} tetrahedron in such a cluster as mentioned in the Theorem, and let an edge $ij\subset t$. There are thus two pentachora $u_1$ and~$u_2$ that have been glued together along~$t$. Each of these has its own functional~$\upphi_{ij}$, and we denote them as $\upphi_{ij}^{(1)}$ and~$\upphi_{ij}^{(2)}$, respectively. It follows then directly from the definition of~$\upphi_{ij}$ that the $t$-component of the sum $\upphi_{ij}^{(1)}+\upphi_{ij}^{(2)}$ \emph{vanishes}. Also, if there are more pentachora in the cluster, they do not possess the 3-face~$t$.

We now define boundary edge functionals as the sum over all pentachora containing the edge:
\begin{equation}\label{pb}
\upphi_{ij}^{\mathrm{boundary}} = \sum_{u\supset ij} \upphi_{ij}^{(u)},
\end{equation}
here $\upphi_{ij}^{(u)}$ is of course the edge functional belonging to pentachoron~$u$. According to the previous paragraph, $\upphi_{ij}^{\mathrm{boundary}}$ does not contain components belonging to inner tetrahedra. Equality~\eqref{phi-b} follows from~\eqref{pb}, and \eqref{gp} is obtained by adding up equalities~\eqref{g-phi} for all pentachora containing vertex~$i$.
\end{proof}

Suppose that, for a cluster of pentachora mentioned in Theorem~\ref{th:bc}, its \emph{boundary} coloring is given satisfying~\eqref{phi-b}, and we want to find all possible corresponding colorings of inner tetrahedra. As they are determined by linear relations, they form an affine subspace in the linear space of all inner edge colorings. Recall that the \emph{dimension} of this subspace can play the role of a \emph{multiplicity}, according to item~\ref{i:md} on page~\pageref{i:md}. The results of an accurate direct calculation of these dimensions for different clusters are presented below.

\paragraph{2~pentachora.} The cluster has 8 boundary 3-faces, hence the space of \emph{all} colorings is 16-dimen\-sional. There are 14 edges, all boundary, hence 14 edge functionals, but 6 linear dependencies~\eqref{gp} between them due to 6 vertices. The dimension of the space of \emph{permitted} colorings is the following alternating sum:
\begin{equation}\label{2p}
2\times 8 - 14 + 6 = 8.
\end{equation}
There is one inner tetrahedron, and its color is determined uniquely from any permitted coloring of boundary tetrahedra.

\begin{remark}
Formula~\eqref{2p} says that there are 14 dependencies between tetrahedron colors and then 6 dependencies between dependencies. We will study the algebraic structure responsible for these and other dependencies---an \emph{exotic chain complex}---below in Section~\ref{s:exotic}.
\end{remark}

\paragraph{3~pentachora.} There are now 9 boundary 3-faces, 15 edges, all boundary, and 6 vertices. So, the dimension of the space of permitted colorings is
\begin{equation}\label{3p}
2\times 9 - 15 + 6 = 9.
\end{equation}
There are three inner tetrahedra, and their colors are determined uniquely from any permitted coloring of boundary tetrahedra.

\paragraph{4~pentachora.} The boundary of the cluster of 4 pentachora is of course the same as in the case of 2 pentachora. So, \eqref{2p} works here as well. The difference is, however, that there is one ``inner degree of freedom'': given a permitted coloring of the boundary, the permitted colorings of the inner tetrahedra form a one-dimen\-sional space.

\paragraph{5~pentachora.} The boundary is the same as the boundary of just one pentachoron. As for ``inner degrees of freedom'', a calculation shows that there are \emph{four} of them.

\subsection{Full hexagon}\label{ss:fh}

According to Subsection~\ref{ss:bP}, the permitted colorings of left- and right-hand sides of Pachner moves depend only on the coefficients~$\gamma_{ij}$ in the \emph{boundary}. This boundary is the same for the l.h.s.\ and r.h.s.\ of any specific move. Combined with the numbers of ``inner degrees of freedom'', also written out in Subsection~\ref{ss:bP}, this gives the following theorem.

\begin{theorem}\label{th:fh}
Full hexagon holds in the ``linear'' situation of Subsection~\ref{ss:p}---that is, where permitted colorings are given by linear relations $\upphi_{ij}=0$ of type~\eqref{12}, with edge functionals~$\upphi_{ij}$ obeying, in their turn, relations~\eqref{g-phi} with generic~$\gamma_{ij}$. The multiplicities, understood as dimensions of affine subspaces, are:
\begin{equation}\label{mlt}
a_1=a_2=a_3=0, \qquad a_4=1, \qquad a_5=4.
\end{equation}
The space of permitted colorings of\/~$\partial\Delta^5$---that is, the l.h.s.\ and r.h.s.\ of a Pachner move together---is 9-dimen\-sional.
\end{theorem}

\begin{proof}
It remains to prove the statement about permitted colorings of~$\partial\Delta^5$. It readily follows if we break $\partial\Delta^5$ into two parts, each of three pentachora. The space of permitted colorings of their common boundary is 9-dimen\-sional according to~\eqref{3p}, and then the colors of other tetrahedra are determined uniquely.
\end{proof}

\section{Exotic chain complex}\label{s:exotic}

\subsection{Linear nonconstant hexagon from edge vectors}\label{ss:vectors}

There is one more construction leading to the same linear nonconstant hexagon as in Section~\ref{s:f} and, in a sense, 'dual' to Subsection~\ref{ss:p}.

Instead of edge functionals, we start now from \emph{edge vectors}~$\uppsi_{ij}$: let there be a simplicial complex~$K$, then each edge~$ij$ in~$K$ produces, by definition, a permitted coloring that is nonzero only on tetrahedra $t\supset ij$:
\begin{equation}\label{psi}
\begin{pmatrix} x_t \\ y_t \end{pmatrix} = \uppsi_{ij}|_t = \begin{pmatrix} \psi_{t,ij}^{(1)} \\[.8ex] \psi_{t,ij}^{(2)} \end{pmatrix}.
\end{equation}
Linear relations, associated with vertices, are imposed on edge vectors, similarly to~\eqref{g-phi}:
\begin{equation}\label{e-psi}
\sum_{\substack{\mathrm{all\; edges\;}ij\\ \mathrm{for\; each\; fixed\;}i}} \eta_{ij} \uppsi_{ij} = 0,
\end{equation}
and there are no other conditions on~$\uppsi_{ij}$.

Quite similarly to~\eqref{gamma-phi}, we can write the restriction of~\eqref{e-psi} onto tetrahedron~$1234$:
\begin{equation}\label{eta-psi}
\begin{pmatrix} \uppsi_{12} & \uppsi_{13} & \uppsi_{14} & \uppsi_{23} & \uppsi_{24} & \uppsi_{34} \end{pmatrix}_{1234}
\begin{pmatrix} \eta_{12} & \eta_{21} & 0 & 0 \\ \eta_{13} & 0 & \eta_{31} & 0 \\ \eta_{14} & 0 & 0 & \eta_{41} \\ 0 & \eta_{23} & \eta_{32} & 0 \\ 0 & \eta_{24} & 0 & \eta_{42} \\ 0 & 0 & \eta_{34} & \eta_{43} \end{pmatrix} = 0.
\end{equation}
Of course, matrix equality~\eqref{eta-psi} looks like a \emph{transposed} version of~\eqref{gamma-phi} because $\uppsi_{ij}|_t$~\eqref{psi} are now columns rather than rows. Now, \eqref{eta-psi} shows that~$\uppsi_{ij}$ are, quite like~$\upphi_{ij}$, determined by relations~\eqref{e-psi} up to linear automorphisms of two-dimen\-sional vector spaces of colors of each tetrahedron~$t$---although we are not writing out the analogue of~\eqref{phi_1234-gen} here.

For one separate pentachoron $K=\Delta^4$, all permitted colorings are by definition linear combinations of vectors~\eqref{psi}. One can see that \eqref{e-psi} ensures, for generic coefficients~$\eta_{ij}$, that these linear combinations form again (like in Subsection~\ref{ss:p}) a five-dimen\-sional subspace in the space of all colorings.

\begin{ir}
There is, however, a difference between \eqref{e-psi} and~\eqref{g-phi}: the sum in~\eqref{e-psi}, if we are considering a general complex~$K$, is \emph{not} restricted to one pentachoron.
\end{ir}

\begin{ir}
Also, no signs like~\eqref{eps} are introduced for edge vectors!
\end{ir}

Let now $K$ be a cluster of 2, 3, 4 or 5 pentachora of the type considered in Theorem~\ref{th:bc}, that is, either an l.h.s.\ or r.h.s.\ of a four-dimen\-sional Pachner move. Clearly, the permitted colorings of its boundary~$\partial K$ come only from boundary edges~$ij$, and are the same for (the boundary $\partial \bar K=\partial K$ of) its complementary cluster~$\bar K$ from the other side of the Pachner move. A direct calculation shows also that the multiplicites---understood again as dimensions of spaces of inner colorings corresponding to a given boundary coloring---are the same as in Subsection~\ref{ss:bP}.

This leads to an analogue of Theorem~\ref{th:fh}: \emph{if the permitted colorings of each pentachoron are defined as linear combinations of vectors~\eqref{psi} obeying~\eqref{e-psi} with generic~$\eta_{ij}$, then full hexagon holds, with the same multiplicites~$a_k$ as in Theorem~\ref{th:fh}.}

\subsection{Exotic chain complex for a single pentachoron}\label{ss:es}

Suppose there is just a single pentachoron~$u$ whose five 3-faces can be colored by elements of a two-dimensional vector space over a filed~$F$, as in~\eqref{tc}. Suppose also that we have chosen a generic five-dimen\-sional linear subspace~$V$ in the ten-dimen\-sional space of all such colorings, and called $V$ the space of \emph{permitted colorings}. For a given edge $ij\subset u$, there are three tetrahedra $t\supset ij$. If we want the colorings of the remaining two tetrahedra be zero, this requirement implies four linear equations on an element of~$V$. Thus, there is a one-dimen\-sional linear space of permitted coloring corresponding to~$ij$; we take a nonzero element from it and call it \emph{edge vector}~$\uppsi_{ij}$.

Consider now the four edge vectors~$\uppsi_{ij}$ for a given vertex~$i$; for instance, let $i=1$. Each of the vectors $\uppsi_{12}$, $\uppsi_{13}$, $\uppsi_{14}$ and~$\uppsi_{15}$ satisfies the two linear relations meaning ``the color of $2345$ is zero'', hence their linear span is $5-2=3$-dimen\-sional. This reasoning implies linear relations~\eqref{e-psi}.

On the other hand, a `dual' reasoning, which is left as an exercise for the interested reader, shows the existence of edge \emph{functionals}~$\upphi_{ij}$ with linear dependencies~\eqref{g-phi}.

As a result, we obtain the following `exotic chain complex' of vector spaces and linear mappings:

\begin{equation}\label{cD}
 0 \rightarrow 
 \begin{pmatrix}\textrm{vertices}\end{pmatrix}  \stackrel{\eta}{\rightarrow}
 \begin{pmatrix}\textrm{edges}\end{pmatrix} \stackrel{\psi}{\rightarrow}
 \begin{pmatrix} 2\times \textrm{tetrahedra}\end{pmatrix} \stackrel{\phi}{\rightarrow}
 \begin{pmatrix}\textrm{edges} \end{pmatrix} \stackrel{\gamma}{\rightarrow}
 \begin{pmatrix}\textrm{vertices} \end{pmatrix} 
 \rightarrow 0 .
\end{equation}
Here, `$\begin{pmatrix}\textrm{vertices}\end{pmatrix}$' means the five-dimen\-sional vector space of formal linear combinations of the pentachoron vertices---that is, vertices form its \emph{basis}. Similarly, `$\begin{pmatrix}\textrm{edges}\end{pmatrix}$' is the ten-dimen\-sional vector space with the edges as its basis, and `$\begin{pmatrix} 2\times \textrm{tetrahedra}\end{pmatrix}$' is the ten-dimen\-sional vector space of the pentachoron colorings, with two basis elements for each tetrahedron. Mapping~$\eta$ is given by matrix elements~$\eta_{ij}$ between vertex~$i$ and edge~$ij=ji$; mapping~$\psi$ is given by matrix elements~$\uppsi_{ij}|t$ between edge~$ij$ and tetrahedron~$t$; mapping~$\phi$ is given by matrix elements~$\upphi_{ij}|t$ between tetrahedron~$t$ and edge~$ij$; and mapping~$\gamma$ is given by matrix elements~$\gamma_{ij}$ between edge~$ij$ and vertex~$i$.

\begin{remark}
A more pedantic approach might involve \emph{conjugate} spaces to our `$\begin{pmatrix}\textrm{vertices}\end{pmatrix}$' and `$\begin{pmatrix}\textrm{edges}\end{pmatrix}$', but these conjugates are anyhow identified with the mentioned spaces themselves as soon as the bases are fixed.
\end{remark}

In the general case (generic space~$V$ of permitted colorings), complex~\eqref{cD} is \emph{acyclic}. This can be shown by a direct checking of matrix ranks for the four involved linear mappings.

\subsection{An exotic chain complex for a simplicial complex}

Let there be now an arbitrary simplicial complex~$K$, and let a number $\eta_{ij}\in F$ be given for each vertex $i\in K$ and edge~$ij$ having~$i$ as one of its ends. If $\eta_{ij}$ are generic, we can build from them edge vectors with components determined, up to a linear automorphism of~$F^2$ in each tetrahedron, by~\eqref{eta-psi}. We can define permitted colorings of each separate pentachoron~$u$ in~$K$ as restrictions onto~$u$ of linear combinations of edge vectors; these permitted colorings form a five-dimen\-sional subspace in the space of all colorings of~$u$ (as in Subsection~\ref{ss:vectors}), denote it~$V_u\subset F^{10}$.

We can now define an interesting chain complex even without introducing `edge functionals'. Instead, we introduce, for each pentachoron~$u$, \emph{some} five linearly independent linear functionals $F^{10}\to F$ in such way that the intersection of their kernels is exactly~$V_u$. We can unite these linear functionals into one linear map $\Phi_u\colon\; F^{10}\to F^5$, and then take the direct sum $\Phi=\bigoplus_u \Phi_u$ over all pentachora. This leads to the following ``exotic complex'':
\begin{equation}\label{c}
 0 \rightarrow 
 \begin{pmatrix}\textrm{vertices}\end{pmatrix} \stackrel{\eta}{\rightarrow}
 \begin{pmatrix}\textrm{edges}\end{pmatrix} \stackrel{\psi}{\rightarrow}
 \begin{pmatrix} 2\times \textrm{tetrahedra}\end{pmatrix} \stackrel{\Phi}{\rightarrow}
 \begin{pmatrix}5\times\textrm{pentachora} \end{pmatrix} \rightarrow 0.
\end{equation}

\begin{ir}
Although complex~\eqref{c} looks already interesting, it must be noted that a quite \emph{different} exotic complex has been proposed in~\cite[Eq.~(11)]{gcoi}, for the `fermionic' case. The algebraic structures behind hexagon relations look largely unexplored.
\end{ir}

\section{``Infinitesimal'' case}\label{s:fi}

\subsection{Edge functionals}

As we have seen in Section~\ref{s:f}, our construction of ``linear nonconstant hexagon'' begins with arbitrary enough (they only must be `generic') quantities~$\gamma_{ij}$---coefficients of linear dependencies~\eqref{g-phi}. There is one remarkable limiting case of this construction---we call it ``infinitesimal'' case---where permitted colorings arise from the following gammas:
\begin{equation}\label{gi}
\gamma_{ij}=\begin{cases} -1+o\cdot b_{ij} \; & \text{if \ } i<j, \\ 1+o\cdot b_{ji} & \text{if \ } i>j, \end{cases} \qquad o\to 0.
\end{equation}
Here $b_{ij}$, \ $i<j$, are some quantities, again `generic', but otherwise arbitrary.

There is of course no problem to understand the limit $o\to 0$ formally even in the case of a finite field characteristic, and derive from~\eqref{gi} the following formulas for our edge functionals~$\upphi_{ij}$. Although our old expressions~\eqref{phi_1234-gen} cannot be applied directly, edge functionals can still be calculated using the terms of relevant orders of smallness in~\eqref{gamma-phi}.

The first remarkable fact is that $\upphi_{ij}$ can be chosen to depend only on the quantities
\begin{equation}\label{ob}
\omega_{ijk} = b_{ij}-b_{ik}+b_{jk}.
\end{equation}
Namely, for tetrahedron $t=1234$ considered as a 3-face of pentachoron $u=12345$, the components of (nonvanishing) edge functionals can be chosen as follows:
\begin{equation}\label{phi_1234-inf}
\setlength\arraycolsep{.6em}
\begin{pmatrix} \upphi_{12} \\ \upphi_{13} \\ \upphi_{14} \\
 \upphi_{23} \\ \upphi_{24} \\ \upphi_{34} \end{pmatrix}_{1234} =
\begin{pmatrix} \omega_{234}-\omega_{134} & 0 \\
                \omega_{124} & \omega_{234} \\
                -\omega_{123} & -\omega_{234} \\
                -\omega_{124} & -\omega_{134} \\
                \omega_{123} & \omega_{134} \\
                0 & \omega_{123}-\omega_{124}
\end{pmatrix}
\end{equation}

For an arbitrary tetrahedron $t=ijkl$, expressions \eqref{phi_1234-inf} apply with the obvious substitution $1\mapsto i, \ldots, 4\mapsto l$, and also the minus sign is added if the orientation of~$t$ given by the increasing order $i<j<k<l$ of its vertices does not coincide with the orientation of pentachoron~$u$ given also by the increasing order of vertices.

\subsection{Dimensions of permitted coloring spaces and full hexagon}\label{ss:d9}

As the gammas~\eqref{gi} in our ``infinitesimal'' case do not look generic, it must be checked separately that everything works well with the permitted colorings determined by edge functionals of type~\eqref{phi_1234-inf}. So, it has been done by direct calculations, and the results are as follows.

First, edge functionals of type~\eqref{phi_1234-inf} give indeed a five-dimen\-sional space of permitted colorings in a separate pentachoron. Then, \emph{all} the dimensions of linear spaces mentioned in Subsections \ref{ss:bP} and~\ref{ss:fh} retain their values. In particular, everything stated in Theorem~\ref{th:fh} stays valid: full hexagon does hold, the dimensions of `inner' colorings with a fixed boundary coloring are as in~\eqref{mlt}, and the space of permitted colorings of~$\Delta^5$---or, which is the same, of~$\partial\Delta^5$---is 9-dimen\-sional.

\subsection{Edge vectors}\label{ss:vi}

For a given pentachoron~$u$, the conditions $\upphi_{ij}=0$ for all edges $ij\subset u$ determine, as usual, a five-dimen\-sional space~$V_u$ of permitted colorings. This can be checked by a direct calculation (we are speaking of course of a general position). Given this~$V_u$, we can introduce also edge \emph{vectors}~$\uppsi_{ij}$ in the style of Subsection~\ref{ss:es}. A remarkable fact, that can be shown by a direct calculation but still lacks a conceptual explanation, is the following explicit form of~$\uppsi_{ij}$, which we write out again in terms of its components corresponding to an example tetrahedron~$1234$:
\begin{multline}\label{psi_1234-inf}
\begin{pmatrix} \uppsi_{12} & \uppsi_{13} & \uppsi_{14} & \uppsi_{23} & \uppsi_{24} & \uppsi_{34} \end{pmatrix}_{1234} \\[.5ex]
\renewcommand*{\arraystretch}{1.5}
\setlength\arraycolsep{.6em}
 = 
\begin{pmatrix}\bigl(\omega_{123} \omega_{124}\bigr)^{-1} & 0\\
-\bigl(\omega_{123} ( \omega_{134}-\omega_{234})\bigr)^{-1} & \bigl(\omega_{134} ( \omega_{134}-\omega_{234})\bigr)^{-1}\\
\bigl(\omega_{124} ( \omega_{134}-\omega_{234})\bigr)^{-1} & -\bigl(\omega_{134} ( \omega_{134}-\omega_{234})\bigr)^{-1}\\
\bigl(\omega_{123} ( \omega_{134}-\omega_{234})\bigr)^{-1} & -\bigl(( \omega_{134}-\omega_{234})  \omega_{234}\bigr)^{-1}\\
-\bigl(\omega_{124} ( \omega_{134}-\omega_{234})\bigr)^{-1} & \bigl(( \omega_{134}-\omega_{234})  \omega_{234}\bigr)^{-1}\\
0 & -\bigl(\omega_{134} \omega_{234}\bigr)^{-1}\end{pmatrix} ^{\mathrm T},
\end{multline}
or, surprisingly, simply as follows, using the column~\eqref{phi_1234-inf} of~$\upphi$'s:
\begin{multline}\label{psi-phi}
\begin{pmatrix} \uppsi_{12} & \uppsi_{13} & \uppsi_{14} & \uppsi_{23} & \uppsi_{24} & \uppsi_{34} \end{pmatrix}_{1234} \\
 = 
\dfrac{1}{\omega_{234}-\omega_{134}}
\begin{pmatrix} \dfrac{1}{\omega_{123}\omega_{124}} & 0 \\ 0 & -\dfrac{1}{\omega_{134}\omega_{234}} \end{pmatrix}
\begin{pmatrix} \upphi_{12} \\ \upphi_{13} \\ \upphi_{14} \\
 \upphi_{23} \\ \upphi_{24} \\ \upphi_{34} \end{pmatrix}_{1234}^{\mathrm T}.
\end{multline}

\subsection{One more linear dependence between edge functionals or edge vectors}\label{ss:b}

This Subsection contains some observations that will be used below in Subsection~\ref{ss:comp}, where we will explain that the relation between~$\omega$~\eqref{ob} and our `infinitesimal' case is essentially the same as the relation between the 2-cocycle~$\omega$ and fermionic pentachoron (\,=\,4-simplex) weights in~\cite{full-nonlinear}.

It follows from \eqref{g-phi} and~\eqref{gi} (by simply putting $o=0$ in~\eqref{gi}) that 
\begin{equation}\label{eps-phi}
\sum_{\substack{\text{all four edges }ij\\ \text{for each fixed }i}} \epsilon_{ij} \upphi_{ij} = 0,
\quad\text{where}\quad \epsilon_{ij} = \begin{cases} -1 & \text{if \ } i<j, \\ 1 & \text{if \ } i>j. \end{cases}
\end{equation}
The edges and vertices in~\eqref{eps-phi} belong to a given pentachoron~$u$.

There is, however, one more linear relation occurring because of a special character of our `infinitesimal' case, namely,
\begin{equation}\label{b-phi}
\sum_{\substack{i<j\\ ij\subset u}} b_{ij} \upphi_{ij} = 0.
\end{equation}
The summing in~\eqref{b-phi} goes over all ten edges of~$u$. Relation~\eqref{b-phi} can be checked using \eqref{ob} and the explicit form of~$\upphi_{ij}$ described in~\eqref{phi_1234-inf} and the paragraph after~\eqref{phi_1234-inf}.

It follows then from the relation~\eqref{psi-phi} between edge functionals and edge \emph{vectors} that similar linear relations hold for edge vectors as well, namely
\begin{equation}\label{eps-psi}
\sum_{\substack{\text{all four edges }ij\\ \text{for each fixed }i}} \epsilon_{ij} \uppsi_{ij} = 0,
\quad\text{where}\quad \epsilon_{ij} = \begin{cases} -1 & \text{if \ } i<j, \\ 1 & \text{if \ } i>j, \end{cases}
\end{equation}
and
\begin{equation}\label{b-psi}
\sum_{\substack{i<j\\ ij\subset u}} b_{ij} \uppsi_{ij} = 0.
\end{equation}

\subsection{A nontrivial symmetric bilinear 3-cocycle}\label{ss:3c}

One property that makes the permitted colorings determined by edge functionals of type~\eqref{phi_1234-inf} so remarkable is the existence of \emph{nontrivial hexagon cocycles}. Namely, recall first our definition~\eqref{bc} of a \emph{bilinear} cochain. We introduce now the following \emph{symmetric} bilinear 3-cochain depending on two permitted colorings, whose $t$-components will be denoted as $\begin{pmatrix}x_t\\ y_t\end{pmatrix}$ and $\begin{pmatrix}x'_t\\ y'_t\end{pmatrix}$, respectively:
\begin{equation}\label{3c}
z_{ijkl} = (\omega_{jkl}-\omega_{ikl}) \bigl(\omega_{ijk}\omega_{ijl}x_{ijkl}x'_{ijkl} - \omega_{ikl}\omega_{jkl}y_{ijkl}y'_{ijkl} \bigr)
\end{equation}

\begin{theorem}\label{th:3c}
Cochain~\eqref{3c} is a nontrivial 3-cocycle, that is, it is not a coboundary, but its coboundary vanishes:
\begin{equation}\label{3cc}
\sum_{t\subset u} \epsilon_t^{(u)} z_t = 0.
\end{equation}
Here $\epsilon_t^{(u)}=\pm 1$ are the same as in~\eqref{eps}.
\end{theorem}

\begin{remark}
Recall (see the beginning of Section~\ref{s:h}) that hexagon cohomology depends on a chosen simplicial complex~$K$. In \eqref{3c} and~\eqref{3cc}, however, everything takes place in each pentachoron~$u$ separately, so it is enough to consider just one pentachoron $K=u=\Delta^4$.
\end{remark}

\begin{proof}[Proof of Theorem~\ref{th:3c}]
That \eqref{3c} is not a coboundary, follows simply from the fact that there are no 2-cochains in our theory, see~\eqref{hcc}. Note now that equalities of type~\eqref{psi-phi} mean that, given a permitted coloring~$\begin{pmatrix}x_t\\ y_t\end{pmatrix}_{t\subset u}$ of pentachoron~$u$ (a linear combination of edge vectors), there is a linear functional (a linear combination of edge functionals) that turns \emph{any} permitted coloring~$\begin{pmatrix}x'_t\\ y'_t\end{pmatrix}_{t\subset u}$ into zero, given exactly by the l.h.s.\ of~\eqref{3cc}! We mean of course that we understand~\eqref{3c} as a linear functional acting on colorings~$\begin{pmatrix}x'_t\\ y'_t\end{pmatrix}_{t\subset u}$.
\end{proof}

Possible applications of 3-cocycle~\eqref{3c} will be discussed in Section~\ref{s:d}.

\subsection{A nontrivial symmetric bilinear 4-cocycle}\label{ss:4cn2}

In this subsection, our simplicial complex is a 5-simplex, $K=\Delta^5$. We are going to show that a nontrivial 4-cocycle on~$\Delta^5$ exists already due to the existence of 3-cocycle~\eqref{3c} and the following \emph{dimensional considerations}.

As is known, there is a $\binom{n}{2}=\frac{n(n-1)}{2}$-dimen\-sional space of symmetric bilinear forms over an $n$-dimen\-sional linear space. Hence, here is what the hexagon coboundary operator~$\delta$ looks like for these forms, around the 4-cochains (where we allowed ourself to omit the obvious words ``linearly independent''):
\begin{equation}\label{d4}
\begin{pmatrix}45\\ \text{\small 3-cochains}\end{pmatrix} \xrightarrow[\rank\delta\le 44]{\delta}
\begin{pmatrix}90\\ \text{\small 4-cochains}\end{pmatrix} \stackrel{\delta}\longrightarrow
\begin{pmatrix}45\\ \text{\small 5-cochains}\end{pmatrix} .
\end{equation}
The detailed explanation of~\eqref{d4} is given in the next paragraph.

There are $\binom{9}{2}=45$ linearly independent symmetric bilinear 5-cochains---because the linear space of permitted colorings of~$\Delta^5$ is 9-dimen\-sional, see Theorem~\ref{th:fh} and Subsection~\ref{ss:d9}. Some of these may be coboundaries of 4-cochains, and these (4-cochains) are $6\times 15=90$ (here 6 is the number of pentachora, and $15=\binom{5}{2}$ is the number of symmetric bilinear cochains for one pentachoron). The 4-cochains must be taken here, however, up to the 4-cocycles, and these include coboundaries of the 3-cochains. The number of 3-cochains is $15\times 3=45$ (here 15 is the number of tetrahedra, and $3=\binom{3}{2}$ is the number of symmetric bilinear cochains for one tetrahedron), but as there is one 3-cocycle~\eqref{3c}, the number of their boundaries is~$\le 44$. So, there remain at least $90-44=46$ cochains taken up to adding a coboundary, while there are only~$45$ (to be exact, $\le 45$) conditions for a cochain to be a cocycle.

Hence, there exists a nontrivial symmetric bilinear 4-cocycle on~$\Delta^5$. We now want to have an explicit expression for a representative of its cohomology class. One natural way to obtain it goes as follows.

Begin with taking the coefficients~$\gamma_{ij}$ in the form~\eqref{gi}, except that let $o$ be \emph{finite} (or, even better, let it be an \emph{indeterminate} over our field~$F$). Consider, in this situation, the 3-cochain~$c^{(3)}$ with components $c_{ijkl}^{(3)}=z_{ijkl}$ given by~\eqref{3c}, where $\omega_{ijk}$ is still given by~\eqref{ob}. Cochain~$c^{(3)}$ is no longer a 3-cocycle, but \emph{its coboundary}~$\delta c^{(3)}$ certainly is a \emph{4-cocycle}. Moreover, $\delta c^{(3)}=0$ for $o=0$.

This suggests us that our desired nontrivial 4-cocycle~$z^{(4)}$ may be found in the form of a limit
\begin{equation}\label{o0}
z^{(4)} = \lim_{o\to 0} \frac{\delta c^{(3)}}{o}.
\end{equation}
Indeed, a calculation shows that \eqref{o0} exists---at least in the formal sense---and gives a nontrivial 4-cocycle. Moreover, the nontriviality can be checked using just one pentachoron~$ijklm$: the component~$z_{ijklm}^{(4)}$ turns out to be linearly independent from all the linear combinations of all symmetric bilinear forms $x_t x'_t$, $x_t y'_t + y_t x'_t$ and $y_t y'_t$ for all the five tetrahedra $t\subset ijklm$, and this holds for all field characteristics (even characteristic~2!).

\begin{remark}
The mentioned symmetric bilinear forms span a 14-dimen\-sional linear space: not 15-dimen\-sional due to the linear dependence~\eqref{3cc}. Together with~\eqref{o0}, they span a 15-dimen\-sional space.
\end{remark}

We do not write out an explicit expression for the 4-cocycle~\eqref{o0}, because it has the following shortcoming: it is expressed in terms of simplicial 1-cochain~$b$ (see~\eqref{gi}), while it would look more interesting to have an expression in terms of simplicial 2-cocycle~$\omega$~\eqref{ob}. Moreover, and in more detail, we would prefer to have it as follows:
\begin{itemize}\itemsep 0pt
 \item for each pentachoron $u\subset \Delta^5$, the explicit expression for the $u$-component of such a cocycle should be given in terms of values $\omega_{ijk}$ and variables $x_t$ and~$y_t$, with triangles~$ijk$ and tetrahedra~$t$ belonging to~$u$,
 \item there should be the \emph{same} expressions for all six $u\subset \Delta^5$---that is, one `general' expression for $u=ijklm$ such that the expression for any specific $u=i_u j_u k_u l_u m_u$ is obtained from it just by the substitution $i\mapsto i_u,\ldots,m\mapsto m_u$ (a similar property holds of course for our 3-cocycle~\eqref{3c}).
\end{itemize}
The problem is, however, that we could \emph{not} (as yet) find any cocycle in the cohomology class of~\eqref{o0} and with such an explicit expression---except the case of characteristic~2, see the next Subsection~\ref{ss:4c2}.

\subsection{A nontrivial symmetric bilinear 4-cocycle in characteristic~2: explicit form}\label{ss:4c2}

The situation looks different in characteristic~2: here is an explicit expression for a nontrivial symmetric bilinear 4-cocycle in terms of simplicial 2-cocycle~$\omega$. We denote it~$\zeta$, and write out its component for pentachoron~$12345$; for an arbitrary $u=ijklm$, one has just to substitute $1\mapsto i,\ldots, 5\mapsto m$:
\begin{multline}\label{4c}
 \zeta_{12345} = 
\frac{{\omega_{123}} {\omega_{124}} \left( {\omega_{125}}+{\omega_{123}}\right)  \left( {\omega_{125}}+{\omega_{124}}\right)  \Omega }{\left( {\omega_{124}}+{\omega_{123}}\right)  \left( {\omega_{134}} {\omega_{235}}+{\omega_{135}} {\omega_{234}}\right) }\left( x_{1245}  x'_{1235} +x_{1235}  x'_{1245} \right) \\
+
\frac{{\omega_{123}} \left( {\omega_{125}}+{\omega_{123}}\right)  \left( {\omega_{125}}+{\omega_{124}}\right)  {\omega_{134}} {\omega_{135}} \left( {\omega_{135}}+{\omega_{134}}\right)  }
{{\omega_{134}} {\omega_{235}}+{\omega_{135}} {\omega_{234}}}\left( x_{1345}  x'_{1235} +x_{1235}  x'_{1345} \right) \\
+
\frac{{\omega_{123}^{2}} \left( {\omega_{125}}+{\omega_{123}}\right)  \left( {\omega_{125}}+{\omega_{124}}\right)  \Omega}{\left( {\omega_{124}}+{\omega_{123}}\right)  \left( {\omega_{134}} {\omega_{235}}+{\omega_{135}} {\omega_{234}}\right) } \, x_{1235}  x'_{1235} \\
+
\frac{{\omega_{124}} \left( {\omega_{125}}+{\omega_{123}}\right)  \left( {\omega_{125}}+{\omega_{124}}\right)  {\omega_{134}} {\omega_{135}} \left( {\omega_{135}}+{\omega_{134}}\right) }
{{\omega_{134}} {\omega_{235}}+{\omega_{135}} {\omega_{234}}} \left( x_{1345}  x'_{1245} +x_{1245}  x'_{1345} \right) \\
+
\frac{{\omega_{124}^{2}} \left( {\omega_{125}}+{\omega_{123}}\right)  \left( {\omega_{125}}+{\omega_{124}}\right)  \Omega }{\left( {\omega_{124}}+{\omega_{123}}\right)  \left( {\omega_{134}} {\omega_{235}}+{\omega_{135}} {\omega_{234}}\right) } \, x_{1245}  x'_{1245} \\
+
\frac{\left( {\omega_{125}}+{\omega_{124}}\right)  {\omega_{134}^{2}} {\omega_{135}^{2}} \left( {\omega_{135}}+{\omega_{134}}\right) }{{\omega_{134}} {\omega_{235}}+{\omega_{135}} {\omega_{234}}} \, x_{1345} x'_{1345} \\
+
\frac{{\omega_{234}} {\omega_{235}} \left( {\omega_{235}}+{\omega_{234}}\right)  \left( \Omega+{\omega_{125}} {\omega_{134}} {\omega_{135}}+{\omega_{124}} {\omega_{134}} {\omega_{135}}\right) }
{{\omega_{134}} {\omega_{235}}+{\omega_{135}} {\omega_{234}}} \, x_{2345}  x'_{2345} 
 \\
 +\omega_{123} \omega_{234} \omega_{235} (\omega_{245} + \omega_{345}) \, x_{2345} x'_{2345} \\
 + (\omega_{245} + \omega_{145}) (\omega_{245} + \omega_{345}) \omega_{245} \omega_{345} \, y_{2345} y'_{2345} ,
\end{multline}
where
\begin{equation}\label{4c'}
\Omega = \omega_{124}\omega_{125}\omega_{135}+\omega_{123}\omega_{125}\omega_{135}+\omega_{124}\omega_{125}\omega_{134}+\omega_{123}\omega_{124}\omega_{134}.
\end{equation}

As soon as the explicit form~\eqref{4c},~\eqref{4c'} has been written, the cocycle property, that is,
\begin{equation}\label{4cc}
\zeta_{12345}+\zeta_{12346}+\zeta_{12356}+\zeta_{12456}+\zeta_{13456}+\zeta_{23456} = 0,
\end{equation}
can be checked directly. The nontriviality can be checked just within pentachoron~$12345$, as we have already explained (after formula~\eqref{o0}).

\begin{remark}
Remember that we are in characteristic~2, where there is no need in any sign changing in formulas like~\eqref{4cc}, as well as \eqref{4c} and~\eqref{4c'}.
\end{remark}

\begin{remark}
Expression~\eqref{4c},~\eqref{4c'} is not claimed to be the most elegant form of a cocycle representing its cohomology class.
\end{remark}

Possible applications of 4-cocycle~\eqref{4c},~\eqref{4c'} will be discussed in Section~\ref{s:d}.

\section{Relation to fermionic Gaussian weights}\label{s:fermionic}

Remarkably, the permitted colorings in our `infinitesimal' case have a simple relationship with \emph{fermionic quasi-Gaussian pentachoron weights} considered in~\cite{full-nonlinear}. Below in this Section, we work in the field $F=\mathbb C$ of complex numbers.

\subsection{Scalar product on pentachoron colorings and permitted colorings as a maximal isotropic subspace}

Let $u$ be a pentachoron, and let $c=\begin{pmatrix}x_t\\ y_t \end{pmatrix}_{t\subset u}$ and $c'=\begin{pmatrix}x'_t\\ y'_t \end{pmatrix}_{t\subset u}$ be its two colorings (permitted or not). We introduce a \emph{scalar product} between them as the l.h.s.\ of~\eqref{3cc}, that is,
\begin{equation}\label{sc-pr}
\langle c, c' \rangle = \sum_{t=ijkl\subset u} \epsilon_t^{(u)} (\omega_{jkl}-\omega_{ikl}) \bigl(\omega_{ijk}\omega_{ijl}x_{ijkl}x'_{ijkl} - \omega_{ikl}\omega_{jkl}y_{ijkl}y'_{ijkl} \bigr).
\end{equation}
Then it follows from~\eqref{3cc} that the permitted pentachoron colorings form an \emph{isotropic linear subspace} in the ten-dimen\-sional complex Euclidean space~$V$ of all colorings. As it is five-dimen\-sional, it is also \emph{maximal}.

\subsection{Building fermionic operators and pentachoron weight}\label{ss:u}

As \eqref{sc-pr} is a sum over five tetrahedra, our complex Euclidean space~$V$ is also the direct sum of five complex Euclidean spaces corresponding to these tetrahedra:
\[
V=\bigoplus_{t\subset u} V_t\,,\quad \text{where}\quad V_t\ni \begin{pmatrix}x_t\\ y_t \end{pmatrix}.
\]

We now choose, in each~$V_t$, a new basis consisting of two \emph{isotropic} vectors; we like to denote them $d_t$ and~$e_t$:
\begin{equation}\label{ddee}
\langle d_t, d_t \rangle = \langle e_t, e_t \rangle = 0.
\end{equation}
We also normalize them so that their scalar product is~$\epsilon_t^{(u)}$:
\begin{equation}\label{de}
\langle d_t, e_t \rangle = \epsilon_t^{(u)}.
\end{equation}

\begin{ir}
Normalization~\eqref{de} means that $d_t=\begin{pmatrix}x_t^d\\ y_t^d \end{pmatrix}$ and $e_t=\begin{pmatrix}x_t^e\\ y_t^e \end{pmatrix}$, \ $t=ijkl$, are such that
\begin{equation}\label{det}
(\omega_{jkl}-\omega_{ikl}) \bigl(\omega_{ijk}\omega_{ijl}x_{ijkl}^d x_{ijkl}^e - \omega_{ikl}\omega_{jkl}y_{ijkl}^d y_{ijkl}^e \bigr) = 1.
\end{equation}
As $\epsilon_t^{(u)}$ does not enter in~\eqref{det}, this condition deals only with the tetrahedron~$t$ and does not depend on the choice of pentachoron~$u\supset t$ (if there are more than one of these). The same applies to~\eqref{ddee}.
\end{ir}

Any element $v\in V$ can be decomposed over the basis vectors $d_t$ and~$e_t$; we like to denote the respective coefficients as $\beta_t^{(u)}$ and~$\epsilon_t^{(u)}\gamma_t^{(u)}$:
\begin{equation}\label{v}
v = \sum_{t\subset u} (\beta_t^{(u)} d_t + \epsilon_t^{(u)}\gamma_t^{(u)} e_t).
\end{equation}
It follows from \eqref{ddee} and~\eqref{de} that the scalar product of two such elements $v$ and~$v'$ is
\begin{equation}\label{vv'}
\langle v, v' \rangle = \sum_{t\subset u} (\beta_t \gamma'_t + \beta'_t \gamma_t),
\end{equation}
where we omitted the superscript~$(u)$ for a moment; instead, we marked by a prime the coefficients belonging to~$v'$. We now note that \eqref{vv'} is nothing but the scalar product~\cite[(5)]{full-nonlinear} for linear Grassmann differential operators of the form~\cite[(4)]{full-nonlinear}, that is, in our case,
\begin{equation}\label{Gl}
\sum_{t\subset u} (\beta_t^{(u)} \partial_t + \gamma_t^{(u)} \vartheta_t).
\end{equation}
In~\eqref{Gl}, we have attached a \emph{Grassmann algebra generator}~$\vartheta_t$ to each tetrahedron~$t$, while $\partial_t = \partial / \partial \vartheta_t $ is the corresponding left differentiation (in~\cite{full-nonlinear}, notation~$x_t$ is used instead of~$\vartheta_t$). The scalar product of two operators~\eqref{Gl} is simply their anticommutator (a scalar operator identified with the corresponding scalar).

\begin{remark}
We don't give here an exposition of the theory of Grassmann algebras and Berezin integral (the latter is needed when we glue pentachora together). The reader can consult the books~\cite{B,B-super}. Or, a very brief account can be found, for instance, in~\cite[Section~2]{full-nonlinear}.
\end{remark}

We now put in correspondence to a permitted pentachoron~$u$ coloring the corresponding Grassmann operator~\eqref{Gl}. Then, a maximal isotropic space of Grassmann operators corresponds to all permitted colorings, and there is, up to a nonvanishing scalar factor, exactly one (nonzero) Grassmann algebra element~$\mathcal W_u$ \emph{annihilated} by any operator in this maximal isotropic space, called `quasi-Gaussian 4-simplex weight' in~\cite{full-nonlinear}.

\begin{remark}
The existence and uniqueness of~$\mathcal W_u$ (up to a nonzero factor) is of course a classical fact, see~\cite{Ch}, where one can find also many other interesting facts about Grassmann algebras. Or, a simple proof suited for our situation can be found in~\cite[Theorem~2]{KS2013}.
\end{remark}

\subsection{Comparing with previously known Grassmann weights parameterized by 2-cocycles}\label{ss:comp}

In~\cite{full-nonlinear}, a Grassmann pentachoron weight has been constructed corresponding to a given 2-cocycle~$\omega$ on the pentachoron, that is, the set of (generic) complex numbers~$\omega_{ijk}$ attached to each triangle~$ijk$, \ $i<j<k$, and such that
\begin{equation}\label{om}
\omega_{ijk}-\omega_{ijl}+\omega_{ikl}-\omega_{jkl}=0.
\end{equation}
We now want to show that our~$\mathcal W_u$ is, essentially, the same as the mentioned weight.

To be exact, recall that the weight in~\cite{full-nonlinear} was determined by~$\omega$ to within \emph{gauge transformations}, see~\cite[Definition~9]{full-nonlinear}. So, what we are going to show is that our~$\mathcal W_u$ is one possible weight within the relevant gauge equivalence class.

\begin{remark}
Surely, \eqref{om} follows from~\eqref{ob}. Also, our 1-cochain~$b$ in \eqref{gi} and~\eqref{ob} is, essentially, the same object as~$\nu$ in~\cite[(27)]{full-nonlinear}.
\end{remark}

According to Subsection~\ref{ss:u}, every permitted coloring~\eqref{v} yields Grassmann differential operator~\eqref{Gl}. Then, a permitted coloring affecting only the tetrahedra containing a given edge---edge vector, in our terminology---yields an \emph{edge operator} in the terminology of~\cite{full-nonlinear}.

We now see that our linear relations \eqref{eps-psi} and~\eqref{b-psi} yield the same relations for edge operators as in~\cite{full-nonlinear} (for the same~$\omega$, of course). Namely, compare \eqref{eps-psi} with the two unnumbered relations between \cite[(25)]{full-nonlinear} and~\cite[(26)]{full-nonlinear}, and \eqref{b-psi} with \cite[(26), (27)]{full-nonlinear}. Finally, it's exactly these linear relations that determine the gauge equivalence class mentioned above, see~\cite[Theorem~4]{full-nonlinear}.

Our specific form of weight~$\mathcal W_u$ implies a specific gauge choice (within the gauge equivalence class). As going into details of the fermionic case is not the aim of the present paper, we only mention here without a proof that such weights work well for Pachner moves: they satisfy the `3--3 relation'~\cite[(53)]{full-nonlinear} and, moreover, their form may be more convenient for calculating the 4-manifold `fermionic' invariant defined in~\cite{gcoi,gcoii} than the form used in those works.

\section{Discussion}\label{s:d}

One possible direction of further work is the generalization of our nonconstant hexagon relations and their cohomologies onto $(n+2)$-gon relations for all integer $n\ge 3$, having in mind their possible applications to piecewise linear $n$-manifolds.

On the other hand, it will be interesting to make direct calculations in the four-dimen\-sional case. The most promising looks the `infinitesimal' case of Section~\ref{s:fi}. We explain briefly what can be expected here.

Suppose we have a \emph{pair} ``PL 4-manifold~$M$, a middle cohomology class $h\in H^2(M,F)$'', for some field~$F$. We can take a triangulation of~$M$, choose a suitable simplicial 2-cocycle~$\omega$ representing class~$h$, and introduce the corresponding linear space~$V_u$ of permitted colorings for each pentachoron~$u$ in the triangulation. We expect that our further constructions, specifically multiplications \eqref{vv3} and~\eqref{vv4}, will depend only on the cohomology class~$h$ of~$\omega$, because similar thing happens in the fermionic case, see~\cite[Theorem~7]{gcoii}.

\paragraph{Possible application of the nontrivial 3-cocycle. }Suppose there are two permitted colorings of our triangulated~$M$ (recall that these are such colorings that their restrictions onto any pentachoron are permitted), whose $t$-components are denoted $\begin{pmatrix}x_t \\ y_t\end{pmatrix}$ and $\begin{pmatrix}x'_t \\ y'_t\end{pmatrix}$, respectively. Then expression~\eqref{3c} can be treated as their \emph{product}, whose result is, due to~\eqref{3cc}, a \emph{simplicial 3-cocycle}. We expect that this will yield a multiplication
\begin{equation}\label{vv3}
V \times V \to H^3(M,F),
\end{equation}
where
\begin{equation}\label{h/p}
V=\Ker \Phi / \LinearImage \psi
\end{equation}
is an `exotic homology' group, see~\eqref{c}. Factorizing by~$\LinearImage \psi$ in~\eqref{h/p} is expected due to the analogy with~\cite[Theorem~5]{cubic}.

\paragraph{Possible application of the nontrivial 4-cocycle. }Similarly, for field~$F$ of characteristic~2, \emph{exotic homology pairing} is expected to arise: $V \times V \to H^4(M,F)$, or, for a connected~$M$,
\begin{equation}\label{vv4}
V \times V \to F.
\end{equation}

\begin{ir}
Recall that the exotic homology itself depends on (the cohomology class of) 2-cocycle~$\omega$!
\end{ir}

\end{document}